\documentclass[11pt,a4paper]{amsart}

\usepackage[T1]{fontenc}
\usepackage{lmodern}
\usepackage[a4paper,margin=30mm]{geometry}
\usepackage{amsmath,amssymb,amsthm,comment}
\usepackage{microtype}
\usepackage[hidelinks]{hyperref}
\usepackage{xcolor}

\hypersetup{
  pdftitle={A direct proof of the equivalence of the A2 and the Helson--Szego conditions},
  pdfauthor={Georgios Psaromiligkos, Michail Paparizos, Christos Grigoriadis}
}

\newcommand{\T}{\mathbb{T}}
\newcommand{\D}{\mathbb{D}}

\DeclareMathOperator{\Real}{Re}
\DeclareMathOperator{\Imag}{Im}
\DeclareMathOperator{\sgn}{sgn}

\numberwithin{equation}{section}

\theoremstyle{plain}
\newtheorem{theorem}{Theorem}[section]
\newtheorem{lemma}[theorem]{Lemma}
\newtheorem{proposition}[theorem]{Proposition}

\theoremstyle{definition}
\newtheorem{definition}[theorem]{Definition}

\theoremstyle{remark}

\AtBeginEnvironment{theorem}{%
}

\title[The $\mathcal{A}_2$ and Helson--Szeg\H{o} conditions]
{A direct proof of the equivalence of the $\mathcal{A}_2$ and the Helson--Szeg\H{o} conditions}

\author[C. Grigoriadis]{Christos Grigoriadis}
\address{Telcontar CG, Chrisanthou Mylona 38, Strovolos, 2014, Nicosia, Cyprus}
\email{grigori7519@gmail.com}

\author[M. Paparizos]{Michail Paparizos}
\address{Department of Mathematics, Michigan State University, East Lansing, Michigan, USA}
\email{paparizo@msu.edu}

\author[G. Psaromiligkos]{Georgios Psaromiligkos}
\address{Department of Mathematics, University of Thessaly, Lamia, Greece}
\email{gpsaromiligkos@uth.gr}

\begin{document}

\begin{abstract}
We give a direct proof of the equivalence between the Muckenhoupt $\mathcal{A}_2$ and the Helson--Szeg\H{o} conditions on the unit circle $\T$. A positive weight $w$ belongs to $\mathcal{A}_2(\T)$ if and only if $\log w=u+Hv$, where $u,v$ are real bounded functions, $\|v\|_\infty<\frac{\pi}{2}$, and $H$ denotes the periodic Hilbert transform. The classical proof combines the Helson--Szeg\H{o} theorem with the Hunt--Muckenhoupt--Wheeden characterization of weighted boundedness of the Hilbert transform. Our proof derives the Helson--Szeg\H{o} representation directly from the Muckenhoupt $\mathcal{A}_2$ condition, without using boundedness of $H$ on $L^2(w)$.
\end{abstract}

\maketitle

\section{Introduction}
\label{sec:introduction}

In 1960, Helson and Szeg\H{o} \cite{HS1960} linked an analytic representation of a positive weight with the boundedness of the conjugate-function operator in the corresponding weighted $L^2$ space. In particular they proved that
\begin{equation}\label{Hilbert ell 2}
\int_{\mathbb{R}}|Hf|^2wdx \leq C \int_{\mathbb{R}}|f|^2wdx ,
\end{equation}
if and only if
\begin{equation}\label{helsz}
\log w = u + Hv, \ \ u, v \in L^\infty,\ \Vert v \Vert_\infty < \frac{\pi}{2}.
\end{equation}

A second description emerged through real-variable weighted inequalities. Muckenhoupt's 1972 work  \cite{M1972} characterized weighted boundedness of the Hardy maximal operator by the $\mathcal{A}_p$ condition. In \cite{HMW1973}, Hunt, Muckenhoupt, and Wheeden proved the corresponding characterization for the conjugate function and the Hilbert transform in 1973. In particular, for $p=2$, they showed that \eqref{Hilbert ell 2} holds if and only if $w$ is an $\mathcal{A}_2$ weight, that is,
\begin{equation}\label{A2 weights}
[w]_{\mathcal{A}_2} =\sup_{I\subset\mathbb{R}} \langle w \rangle_I \langle w^{-1} \rangle_I :=\sup_{I\subset\mathbb{R}}\ \left(\frac{1}{|I|}\int_I w dx\right)\left(\frac{1}{|I|}\int_I \frac{1}{w} dx\right)<\infty.
\end{equation}

Consequently, the \(\mathcal{A}_2\) condition and the Helson-Szeg\H{o} condition are equivalent. However, the known proof of this equivalence proceeds through the characterization of the weighted \(L^2\)-boundedness of the Hilbert transform \eqref{Hilbert ell 2}. To the best of our knowledge, no direct proof of the implication from the $\mathcal{A}_2$ condition to the Helson--Szeg\H{o} representation, avoiding boundedness of $H$ on $L^2(w)$, has previously appeared in the literature. The purpose of this paper is to provide such a proof.

We will work on the unit circle $\mathbb{T}$, so we adjust the conditions: A measurable function $0<w<\infty$ almost everywhere belongs to $\mathcal{A}_2(\T)$ if $w,w^{-1}\in L^1(\T)$ and
\begin{displaymath}
[w]_{\mathcal{A}_2}
=
\sup_I\langle w\rangle_I\langle w^{-1}\rangle_I
<
\infty.
\end{displaymath}
The supremum is over all arcs, including the whole circle. We give a direct proof of the following:
\begin{theorem}
\label{main}
A measurable function $0<w<\infty$ almost everywhere belongs to $\mathcal{A}_2(\T)$ if and only if there are $u,v\in L^\infty(\T)$ such that
$$
w=e^{u+Hv}
\quad\text{almost everywhere and }
\quad
\|v\|_\infty<\frac{\pi}{2}.
$$
\end{theorem}

\subsection{Notation and conventions} The circle and real-line formulations are related by the Cayley transform $C(z)=\frac{z-i}{z+i}$. The $\mathcal{A}_2$ conditions correspond under composition with $C$, as follows from their equivalent Poisson formulations. Harmonic conjugates correspond up to an additive constant, which is absorbed into $u$. All integrals and norms on $\T$ use normalized arc measure
\begin{displaymath}
dm(e^{it})=\frac{dt}{2\pi},
\qquad
|I|=m(I),
\qquad
\langle g\rangle_I=\frac{1}{|I|}\int_I g\,dm.
\end{displaymath}
For $g\in L^1(\T)$, its Poisson extension is
\begin{displaymath}
Pg(z)=\int_\T P_z(\zeta)g(\zeta)\,dm(\zeta),
\ \ \text{ where }\ \ 
P_z(\zeta)=\frac{1-|z|^2}{|\zeta-z|^2}.
\end{displaymath}
The periodic Hilbert transform on unweighted $L^2(\T)$ is defined by
\begin{displaymath}
\widehat{Hg}(n)=-i\sgn(n)\widehat g(n),
\qquad
\sgn(0)=0.
\end{displaymath}
Also, $H^*=-H$ follows directly from Parseval.

\subsection{Historical background}

As recalled in the introduction, Helson and Szeg\H{o} characterized the boundedness of $H$ on $L^2(w)$ through their analytic representation of the weight \cite{HS1960}, while Hunt, Muckenhoupt, and Wheeden characterized the same boundedness through the $\mathcal{A}_2$ condition \cite{HMW1973}. Together, these results give the equivalence in Theorem \ref{main}. Their original proofs were lengthy and technically demanding, so obtaining the equivalence in this way required combining two substantial arguments.

Treil and Volberg later gave a simpler complex-analytic proof of the Hunt--Muckenhoupt--Wheeden characterization at $p=2$: a weight belongs to $\mathcal{A}_2(\T)$ if and only if $H$ is bounded on $L^2(w)$ \cite{TV1997CF}. Their argument uses harmonic extensions and a weighted Littlewood--Paley norm equivalence. This shortened one of the two parts of the classical argument, while retaining weighted operator boundedness as the intermediate step between the two descriptions of the weight.

The question of passing directly between these descriptions had already been discussed explicitly in Garc\'ia-Cuerva's survey \cite{GC1982}. He gave a direct argument showing that the Helson--Szeg\H{o} condition implies $\mathcal{A}_2$ and noted that a direct proof of the converse was not known. The absence of a direct proof of the equivalence is also emphasized by Garnett \cite[Notes to Chapter VI, p.~260]{G2007}.

Quantitative questions form a further part of the history. Petermichl proved the sharp $\mathcal A_p$ dependence for the Hilbert transform using dyadic representations and Bellman-function estimates \cite{P2007}, and Hyt\"onen established the sharp $\mathcal{A}_2$ bound for general Calder\'on--Zygmund operators \cite{H2012}. Lerner gave a simpler proof of the latter result using local mean oscillation and positive sparse averaging operators \cite{L2013}. Hyt\"onen, Lacey, Martikainen, Orponen, Reguera, Sawyer, and Uriarte-Tuero obtained sharp weak and strong estimates for maximal truncations \cite{HLMORSU2012}. In the two-weight problem, Lacey, Sawyer, Shen, and Uriarte-Tuero developed the global-to-local part of a real-variable characterization of the Hilbert-transform inequality \cite{LSSU2014}. These quantitative and two-weight results concern extensions of the operator problem and form part of the broader setting of the present question.

\subsection{The main estimate and the structure of the proof}

Our argument is motivated by Pichorides' inequality \cite[Theorem~2.4]{Pi1972}. For a nonnegative function $g\in L\log^+L(\T)$ normalized by $\int_\T g\,dm=1$, it gives $\int_\T g\log g\,dm\le\frac{\pi}{2}\|Hg\|_1$. Combined with exponential integrability of $\log w$, this suggests looking for a logarithmic pairing estimate with the critical coefficient $\frac{\pi}{2}$. 

The estimate needed here differs in both its object and its hypotheses. Pichorides controls the entropy of a nonnegative test function; we control the linear pairing $\int_\T(\log w)f\,dm$, where $w\in\mathcal{A}_2$ and $f$ is a real mean-zero trigonometric polynomial that may change sign. Applying his inequality to $|f|$ introduces $H(|f|)$, which does not provide the required control in terms of $Hf$. We instead prove the signed estimate
\begin{equation}\label{intro ineq}
\bigg|\int_\T f \log w\,dm\bigg|
\le
M\|f\|_1+
\bigg(\frac{\pi}{2}+\eta\bigg)\|Hf\|_1,
\end{equation}
for every $\eta>0$, with $M$ depending only on $[w]_{\mathcal{A}_2}$ and $\eta$. The estimate is therefore tailored to the logarithm of an $\mathcal{A}_2$ weight and preserves the signs of the test function.

The main new analytic ingredient is the gluing in Lemma \ref{gluing}. Similar ideas of joining harmonic pieces across the imaginary axis already appear in Pichorides' proof \cite[Lemma~3.5]{Pi1972} and in the harmonic-spline method of Ess\'en, Shea, and Stanton \cite[Section~2]{ESS2002}.

Self-improvement of $\mathcal{A}_2$ then supplies a coefficient strictly below $\frac{\pi}{2}$, after which Hahn--Banach yields the Helson--Szeg\H{o} representation.

Section \ref{sec:weights} establishes the main tools to prove Theorem \ref{main}. We first give the weight estimates needed in the argument. We then give a criterion for superharmonicity, which allows us to prove inequality \eqref{intro ineq}. Theorem \ref{main} is completed in Section \ref{sec:forward} using Hahn-Banach to prove the forward part. We also include the already known result, following Garc\'{i}a-Cuerva, for completeness of the argument.

\section*{Acknowledgements}

The authors would like to thank Ignacio Uriarte-Tuero for bringing this problem to their attention.

\section{Set-up and Preparation} \label{sec:weights}

\begin{lemma}\label{self-improvement}
Let $w$ be an $\mathcal{A}_2$ weight with $[w]_{\mathcal{A}_2}<\infty$. For 
$\displaystyle
\rho=1-\frac{1}{4[w]_{\mathcal{A}_2}}
$ 
and
$\delta=-\dfrac{\log\rho}{2\log4}>0$, the weight $w^{1+\delta}\in \mathcal{A}_2$. More precisely, one has
$$
\frac{1}{|I|}\int_I w^{1+\delta}\,dm
\le
D\langle w\rangle_I^{1+\delta}
\ \ \text{ and }\ \ 
[w^{1+\delta}]_{\mathcal{A}_2}\le D^2[w]_{\mathcal{A}_2}^{1+\delta}.
$$
for 
$D=\dfrac{2^\delta}{1-4^\delta\rho}$. The same reverse H\"older estimate holds for $w^{-1}$.
\end{lemma}

\begin{proof}
Let $\sigma:=w^{-1}$. Cauchy-Schwarz inequality gives that $[w]_{\mathcal{A}_2}\ge1$, so $0<\rho<1$ and $\delta>0$.

Let $E\subset I$ with $|E|\le\dfrac{|I|}{2}$. By Cauchy-Schwarz inequality,
$$
\frac{|I|^2}{4}
\ \le \ |I\setminus E|^2 
\ \le \ w(I\setminus E)\sigma(I)
\ \le \ [w]_{\mathcal{A}_2}|I|^2\frac{w(I\setminus E)}{w(I)}.
$$
Dividing by $|I|^2$ and rearranging we get $w(I\setminus E)\ge\dfrac{w(I)}{4[w]_{\mathcal{A}_2}}$, which implies $$w(E)\le\rho w(I).$$

The condition is symmetric in $w$ and $\sigma$, namely $\sigma(E)\le\rho \sigma(I)$ whenever $|E|\le\frac{|I|}{2}$.

Next, we fix an arc $I$ and use its dyadic subdivisions. We perform a classical stopping-time argument. Inside a selected dyadic arc $J$, select all maximal proper dyadic subarcs $Q$ satisfying
\begin{displaymath}
\langle w\rangle_Q>2\langle w\rangle_J.
\end{displaymath}
They are pairwise disjoint. Summing the stopping inequality over them gives
\begin{displaymath}
2\langle w\rangle_J\sum_Q|Q|
\le
\sum_Q w(Q)
\le
w(J),
\end{displaymath}
and hence $\displaystyle \sum_Q|Q|\le\frac{|J|}{2}$. Therefore, for $G:= \bigcup\limits_Q Q$, we have $|G|\le \dfrac{|J|}{2}$ and by the work above we get $$w(G) \le \rho w(J)$$

Next, let $\widetilde Q$ be the dyadic parent of a selected arc $Q$. If $\widetilde Q\ne J$, maximality means that its average does not exceed $2\langle w\rangle_J$. The same upper bound holds when $\widetilde Q=J$. Since a dyadic child has half the length of its parent,
\begin{displaymath}
\langle w\rangle_Q
\le
2\langle w\rangle_{\widetilde Q}
\le
4\langle w\rangle_J.
\end{displaymath}

Now, let $x\in J\setminus G$. On generation $n$, let $R_n$ be the dyadic arc which contains $x$. Since it is not maximal, we get 
$$
\langle w\rangle_{R_n} \leq 2\langle w\rangle_J.
$$
Dyadic differentiation therefore gives $w\le2\langle w\rangle_J$ almost everywhere there.

We now repeat this selection inside every selected arc. Let $G_0=I$, and let $G_n$ be the union of the maximal arcs in generation $n$. The arcs within each generation are disjoint, and the previous estimates imply
\begin{equation}\label{union}
w(G_n)\le\rho^n w(I)
\end{equation}
and 
$\langle w\rangle_Q\le4^n\langle w\rangle_I
\text{ for each generation-$n$ arc }Q.
$
Consequently,
$$
w\le2\cdot4^n\langle w\rangle_I
\quad
\text{almost everywhere on }G_n\setminus G_{n+1}.
$$
By inequality \eqref{union}, we deduce that the intersection $\bigcap_nG_n$ has $w$-measure zero. Because $w>0$ almost everywhere, it also has arc measure zero. The disjoint sets $G_n\setminus G_{n+1}$ therefore cover $I$ up to a null set. Therefore,
\begin{displaymath}
\begin{aligned}
\int_I w^{1+\delta}\,dm
&=
\sum_{n\ge0}\int_{G_n\setminus G_{n+1}}w^\delta w\,dm\\
&\le
\sum_{n\ge0}
(2\cdot4^n\langle w\rangle_I)^\delta w(G_n)\\
&\le
2^\delta\langle w\rangle_I^\delta w(I)
\sum_{n\ge0}(4^\delta\rho)^n.
\end{aligned}
\end{displaymath}
Our choice of $\delta$ gives $4^\delta\rho=\sqrt{\rho}<1$, so the last expression equals $D|I|\langle w\rangle_I^{1+\delta}$.

By applying the identical argument to $\sigma$ and multiplying the two reverse H\"older estimates,
$$
\langle w^{1+\delta}\rangle_I\langle \sigma^{{1+\delta}}\rangle_I
\le
D^2
\left(\langle w\rangle_I\langle \sigma\rangle_I \right)^{1+\delta}
\le
D^2[w]_{\mathcal{A}_2}^{1+\delta}.
$$
Taking $I=\mathbb{T}$, the same result also follows. 
\end{proof}

\begin{lemma}
\label{poisson-A2}
Let $w$ be an $\mathcal{A}_2$ weight with $[w]_{\mathcal{A}_2}<\infty$. Then
\begin{displaymath}
1\le Pw(z)P(w^{-1})(z)\le\frac{512}{3}[w]_{\mathcal{A}_2}^2,
\qquad
z\in\D.
\end{displaymath}
Conversely, if $w,w^{-1}\in L^1(\T)$ and the above Poisson product is bounded by $K$ on $\D$, then
\begin{displaymath}
[w]_{\mathcal{A}_2}\le(1+\pi^2)^2K.
\end{displaymath}
\end{lemma}

\begin{proof}
The measure $P_z\,dm$ is a probability measure. Cauchy--Schwarz with respect to it gives
$$
1=
\bigg(\int_\T w^{1/2}w^{-1/2}P_z\,dm\bigg)^2
\le
Pw(z)P(w^{-1})(z).
$$
For $|z|\le\frac{1}{2}$ one has $P_z\le3$. It follows that $Pw(z)P(w^{-1})(z)\le9[w]_{\mathcal{A}_2}$. Now let $z=re^{it_0}$ with $r>\frac{1}{2}$, and set $d=1-r$. Choose $N\ge1$ such that
$$
2^{-N-1}<d\le2^{-N}.
$$
For $0\le k\le N$, let $I_k$ be the arc centered at $e^{it_0}$ of normalized length $\ell_k=2^{k-N}$. Thus $I_N=\T$ and each arc $I_k=\{e^{i(t_0+2\pi t)}:|t|\le\frac{\ell_k}{2}\}$ is twice the length of the preceding one. 

For each point $\zeta=e^{i(t_0+2\pi t)}, -\frac{1}{2}\le t<\frac{1}{2}$ of the circle and since $z=re^{it_0}$ and $d=1-r$, we have
\begin{displaymath}
|\zeta-z|^2
=|e^{2\pi it}-r|^2
=d^2+4r\sin^2(\pi t).
\end{displaymath}
For $|t|\le\frac{1}{2}$, the inequality $\sin(\pi|t|)\ge2|t|$, together with $r>\frac{1}{2}$, gives $|\zeta-z|^2\ge d^2+8t^2$. 
Therefore, using $1-r^2\le2d$, we obtain
\begin{displaymath}
P_z(\zeta)
=\frac{1-r^2}{|\zeta-z|^2}
\le\frac{2d}{d^2+8t^2}.
\end{displaymath}

For $\zeta \in I_0$ we have $P_z(\zeta)\leq \frac{2}{d}<\frac{4}{\ell_0}$ by the choice of $N$ above. On $I_k\setminus I_{k-1}$, one has $|t|\ge\frac{\ell_k}{4}$, so
$$
P_z
\le
\frac{2d}{8t^2}
\le
\frac{4d}{\ell_k^2}
\le
\frac{4\cdot2^{-k}}{\ell_k}.
$$
Therefore,
\begin{equation}\label{poison}
Pg(z)=
\sum_{k=0}^{N}\int_{I_k\setminus I_{k-1}}P_z(\zeta)g(\zeta)dm(\zeta)
\le
4\sum_{k=0}^N2^{-k}\langle g\rangle_{I_k}
\end{equation}
for every nonnegative integrable function $g$.

Cauchy-Schwarz inequality gives 
$$
|I_{j+1}\setminus I_j|^2\le w(I_{j+1}\setminus I_j)\frac{[w]_{\mathcal{A}_2}|I_{j+1}|^2}{w(I_{j+1})}
$$
which implies that $w(I_{j+1}\setminus I_j)\ge \frac{1}{4A}$, since $|I_j|=|I_{j+1}\setminus I_j|=\frac{|I_{j+1}|}{2}$. Thus, 
$$
w(I_k)=w(I_{k+1})-w(I_{k+1}\setminus I_k) \le \rho w(I_{j+1})
$$
where $\rho$ is the same as in Lemma \ref{self-improvement}. Repeating this inequality,
$$
w(I_k)\le\rho^{l-k}w(I_l)
\qquad
(l\ge k).
$$
For $\sigma=w^{-1}$, it follows that
$$
\langle w\rangle_{I_k}\langle\sigma\rangle_{I_l}
\ \ \le \ \
\rho^{l-k}
\frac{w(I_l)\sigma(I_l)}{|I_k||I_l|}
\ \ \le \ \
[w]_{\mathcal{A}_2}\rho^{l-k}\frac{|I_l|}{|I_k|}
\ \  = \ \
[w]_{\mathcal{A}_2}(2\rho)^{l-k}.
$$
For $k\ge l$, use $\sigma(I_l)\le\rho^{k-l}\sigma(I_k)$ and estimate as before.

Inequality \eqref{poison} for $w$ and $\sigma$ combined with the estimates above gives
$$
Pw(z)P\sigma(z)
\le
16\sum_{k,l\ge0}^N 2^{-k-l}\langle w\rangle_{I_k}\langle\sigma\rangle_{I_l}
\le
16[w]_{\mathcal{A}_2}\sum_{k,l\ge0}
4^{-\min(k,l)}\rho^{|k-l|}.
$$
To evaluate the sum, separate the diagonal and the two off-diagonal regions:
$$
\sum_{k,l\ge0}4^{-\min(k,l)}\rho^{|k-l|}
=
\sum_{k\ge0}4^{-k}
+
2\sum_{k\ge0}\sum_{j\ge1}4^{-k}\rho^j
=
\frac{4}{3}\frac{1+\rho}{1-\rho}.
$$
Since $\frac{1+\rho}{1-\rho}\le8A$, we conclude that
$$
Pw(z)P\sigma(z)
\le
\frac{512}{3}[w]_{\mathcal{A}_2}^2
$$

For the converse, let $I=\{e^{i(t_I+2\pi t)}: |t|\leq\frac{\ell}{2}\}$ for  some $\ell<\frac{1}{2}$ with midpoint $e^{it_I}$ and set
\begin{displaymath}
z_I=(1-\ell)e^{it_I}.
\end{displaymath}
For each $\zeta\in I$,
$$
|\zeta-z_I|^2
=\ell^2+4(1-\ell)\sin^2(\pi t)
\le
\ell^2+4(1-\ell)\frac{\pi^2\ell^2}{4}
\le(1+\pi^2)\ell^2.
$$
Also, $1-|z_I|^2=2\ell-\ell^2\ge\ell$. Thus the Poisson kernel satisfies
$$
P_{z_I}(\zeta)
=\frac{1-|z_I|^2}{|\zeta-z_I|^2}
\ge\frac{1}{(1+\pi^2)\ell},
\qquad \zeta\in I.
$$

Integrating against $w$ and $w^{-1}$ gives
$$
\langle w\rangle_I\le(1+\pi^2)Pw(z_I),
\qquad
\langle w^{-1}\rangle_I\le(1+\pi^2)P(w^{-1})(z_I).
$$
Multiplying these inequalities and using the hypothesis that the Poisson product is bounded, we obtain
\begin{displaymath}
\langle w\rangle_I\langle w^{-1}\rangle_I
\le(1+\pi^2)^2Pw(z_I)P(w^{-1})(z_I)
\le(1+\pi^2)^2K.
\end{displaymath}

If $\ell>\frac{1}{2}$,
\begin{displaymath}
\langle w\rangle_I\langle w^{-1}\rangle_I
\le\frac{1}{\ell^2}
\bigg(\int_{\mathbb T}w\,dm\bigg)
\bigg(\int_{\mathbb T}w^{-1}\,dm\bigg)
\le4Pw(0)P(w^{-1})(0)
\le4K.
\end{displaymath}
This proves the required $\mathcal{A}_2$ bound for every arc.
\end{proof}

\begin{definition}
A continuous real function $B$ on an open set $\Omega\subset\mathbb C$ is superharmonic if it has the harmonic comparison property: on every disk compactly contained in $\Omega$, a harmonic function continuous on the closed disk and bounded above by $B$ on the boundary is bounded above by $B$ throughout the disk.
\end{definition}

We shall obtain this property from harmonic functions $S$ that touch $B$ from above. More precisely, 

\begin{definition}
A \emph{local harmonic upper support} touching $B$ at $z_0\in\Omega$ is a real-valued harmonic function $S$ on an open neighborhood $U\subset\Omega$ of $z_0$ satisfying
$$
B(z_0)=S(z_0)
\ \ \text{ and }\ \ 
B(z)\le S(z)\ \text{for every }z\in U.
$$
We say that $B$ has local harmonic upper support at $z_0$ if there is a local harmonic upper support touching $B$ at $z_0$. 
\end{definition}

\begin{lemma}\label{upper-supports}
Let $\Omega\subset\mathbb C$ be open and let $B:\Omega\to\mathbb R$ be continuous. Suppose that $B$ has a local harmonic upper support at every point of $\Omega$. Then $B$ is superharmonic. In particular, for every disk $D(c,R)$ with $\overline{D(c,R)}\subset\Omega$,
\begin{equation}\label{supermean}
\frac{1}{2\pi}\int_0^{2\pi}B(c+Re^{i\theta})\,d\theta\le B(c).
\end{equation}
\end{lemma}

\begin{proof}
Fix $\overline{D(c,R)}\subset\Omega$. Let $q$ be continuous on the closed disk, harmonic in its interior, and satisfy $q\le B$ on the boundary. We will prove that $q\le B$ throughout $D(c,R)$, which would imply that $B$ is superharmonic, by contradiction.

To this end, suppose that $q>B$ at an interior point. For sufficiently small $\varepsilon>0$, the function
\begin{displaymath}
q_\varepsilon(z)=q(z)+\varepsilon(|z-c|^2-R^2)
\end{displaymath}
still exceeds $B$ somewhere and agrees with $q$ on the boundary. Consequently, $q_\varepsilon-B$ attains a positive maximum at an interior point $z_0$.

Let $S$ be a harmonic upper support touching $B$ at $z_0$. Near that point,
$$
q_\varepsilon-S
\le
q_\varepsilon-B
\le
q_\varepsilon(z_0)-B(z_0)
=
q_\varepsilon(z_0)-S(z_0).
$$
Hence $q_\varepsilon-S$ has a local maximum at $z_0$. Choose $r>0$ sufficiently small so that $q-S$ is harmonic on a neighborhood of $\overline{D(z_0,r)}$. The local maximum property for $q_\varepsilon-S$ gives
$$
(q_\varepsilon-S)(z_0+re^{i\theta})
\le(q_\varepsilon-S)(z_0),
\qquad 0\le\theta\le2\pi.
$$
By the mean-value property for harmonic functions,
\begin{displaymath}
\frac{1}{2\pi}\int_0^{2\pi}
(q-S)(z_0+re^{i\theta})\,d\theta
=(q-S)(z_0).
\end{displaymath}
Moreover,
\begin{displaymath}
\frac{1}{2\pi}\int_0^{2\pi}
|z_0+re^{i\theta}-c|^2\,d\theta
=|z_0-c|^2+r^2.
\end{displaymath}
Since $q_\varepsilon(z)=q(z)+\varepsilon(|z-c|^2-R^2)$, these identities imply
\begin{displaymath}
\frac{1}{2\pi}\int_0^{2\pi}
(q_\varepsilon-S)(z_0+re^{i\theta})\,d\theta
=(q_\varepsilon-S)(z_0)+\varepsilon r^2 > (q_\varepsilon-S)(z_0),
\end{displaymath}
which is a contradiction to the local maximum property for $q_\varepsilon-S$ whic implies that the latter integrand is bounded above by $(q_\varepsilon-S)(z_0)$ at every point of the circle. 

It follows that $q\le B$ throughout $D(c,R)$. Since the disk and the harmonic function $q$ were arbitrary, $B$ has the harmonic comparison property and is therefore superharmonic.

To obtain inequality \eqref{supermean}, let $q$ be the harmonic extension of the continuous boundary values of $B$ on $\partial D(c,R)$, given by the Poisson formula. Then $q$ is continuous on $\overline{D(c,R)}$, harmonic in $D(c,R)$, and satisfies $q=B$ on the boundary. The mean-value property yields
\begin{displaymath}
q(c)
=\frac{1}{2\pi}\int_0^{2\pi}q(c+Re^{i\theta})\,d\theta
=\frac{1}{2\pi}\int_0^{2\pi}B(c+Re^{i\theta})\,d\theta.
\end{displaymath}
Consequently, since $q(c)\le B(c)$, 
\begin{displaymath}
\frac{1}{2\pi}\int_0^{2\pi}B(c+Re^{i\theta})\,d\theta\le B(c),
\end{displaymath}
as claimed.
\end{proof}

For $R>0$ and $\zeta=p+iq$ with $p>0$, define
\begin{displaymath}
\mathcal E(R;\zeta)
=
p\log\frac{R}{p}
+
\Real(\zeta\log\zeta).
\end{displaymath}
The logarithm here is the holomorphic branch on the right half-plane with argument in $(-\frac{\pi}{2},\frac{\pi}{2})$.

\begin{lemma}
\label{entropy-lemma}
If $R$ is positive harmonic and $Z=p+iq$ is holomorphic with $p=\Real Z>0$, then $\mathcal E(R;Z)$ has a local harmonic upper support at every point and is superharmonic. Moreover,
\begin{equation}\label{pointwise estimates}
p\log R-\frac{\pi}{2}|q|
\le
\mathcal E(R;p+iq)
\le
p\log R.
\end{equation}
The continuous extension to the imaginary axis satisfies
\begin{displaymath}
\mathcal E(R;iq)=-\frac{\pi}{2}|q| \ \
\text{ and }\ \ 
\mathcal E(R;0)=0.
\end{displaymath}
\end{lemma}

\begin{proof}
Fix a point $z_0$, and set $R_0=R(z_0)$ and $p_0=\Real Z(z_0)$. Both are positive. The elementary inequality $\log t\le t-1$ to $t=\dfrac{p_0R}{R_0p}$ gives, after multiplication by $p$,
$$
p\log\frac{R}{p}
\le
\frac{p_0}{R_0}R
+
\bigg(\log\frac{R_0}{p_0}-1\bigg)p.
$$
Equality holds at $z_0$. Therefore, adding $\Real(Z\log\ Z)$, we get that $\mathcal E(R(z);Z(z))\le S(z)$, where
\begin{displaymath}
S=
\frac{p_0}{R_0}R
+
\bigg(\log\frac{R_0}{p_0}-1\bigg)\Real Z
+
\Real(Z\log Z)
\end{displaymath}
is harmonic. Indeed, its coefficients are fixed real numbers, $R$ and $\operatorname{Re}Z$ are harmonic, and $\operatorname{Re}(Z\log Z)$ is harmonic because $Z\log Z$ is holomorphic. Note also that
$$
\mathcal E(R(z_0);Z(z_0))=S(z_0).
$$
Thus $S$ is a local harmonic upper support for $\mathcal E(R;Z)$ at $z_0$. Since $z_0$ was arbitrary and $\mathcal E(R;Z)$ is continuous, Lemma \ref{upper-supports} implies that $\mathcal E(R;Z)$ is superharmonic.

To establish the inequalities \eqref{pointwise estimates}, expansion in real coordinates gives
\begin{equation}\label{epsilon}
\mathcal E(R;p+iq)-p\log R
=
\frac{p}{2}\log\bigg(1+\frac{q^2}{p^2}\bigg)
-
q\arctan\frac{q}{p}.
\end{equation}
The logarithmic term is nonnegative, and
$|\arctan(\frac{q}{p})|\le \frac{\pi}{2}$. This proves the lower bound.

For the upper bound, define
\begin{displaymath}
\psi(t)=\frac{1}{2}\log(1+t^2)-t\arctan t.
\end{displaymath}
This function is even, $\psi(0)=0$, and $\psi'(t)=-\arctan t\le0$ for $t\ge0$. Hence $\psi(t)\le0$ for all real $t$, and the right-side of \eqref{epsilon} equals $p\psi(\frac{q}{p})\le0$.

To check the boundary extension, let $R$ stay in a compact subset of $(0,\infty)$, as it does locally when $R$ is positive harmonic. First, let $p$ decrease to zero while $q$ tends to a fixed $q_0\ne0$. Then
\begin{displaymath}
p\log\bigg(1+\frac{q^2}{p^2}\bigg)\longrightarrow0
\ \ \text{ and }\ \ 
q\arctan\frac{q}{p}\longrightarrow \frac{\pi}{2}|q_0|.
\end{displaymath}
Thus $\mathcal E(R;p+iq)$ tends to $-\frac{\pi}{2}|q_0|$. At the origin, \eqref{pointwise estimates} yields
$$
|\mathcal E(R;p+iq)|
\le
p|\log R|+\frac{\pi}{2}|q|,
$$
which tends to zero locally uniformly when $R$ stays bounded above and away from zero. The extension is therefore continuous, including at the origin.
\end{proof}

\begin{lemma}\label{gluing}
Let $U,V$ be positive harmonic functions and $h$ a real harmonic function on $\D$ satisfying
$$
-\log V\le h\le\log U
\text{ and }\ 
1\le UV\le K
$$ for a finite constant $K$.
For every $\eta>0$ there is $M=M(K,\eta)>0$ such that, for every holomorphic $F(z)=x(z)+iy(z)$ on $\D$, there is a continuous superharmonic function $B$ with
\begin{equation}\label{lower bnd}
B\ge hx-M|x|-(\frac{\pi}{2}+\eta)|y|,
\end{equation}
with $B(z)=0$ whenever $F(z)=0$.

Here $M$ is independent of $U,V,h,F$ and depends only on $K$ and $\eta$.
\end{lemma}

{\it{Remark:} }
The two harmonic functions $U$ and $V$ play complementary roles. An expression involving $U$ gives the correct lower bound on the side $\Real F>0$, while one involving $V$ does the same on the side $\Real F<0$. These expressions have matching values on the imaginary axis, but matching values alone do not establish superharmonicity across that axis. We therefore introduce a second expression in a narrow cone around it. The value $-\frac{\pi}{2}|q|$ on the imaginary axis is responsible for the coefficient $\frac{\pi}{2}$ in the later conjugate-function estimate. The construction below establishes the particular signed estimate needed here.

\begin{proof}
Choose $0<\varepsilon<\frac{1}{4}$ so small that $\varepsilon\log K+d_\varepsilon<\eta$, where
$$
d_\varepsilon=
\arctan(2\varepsilon)
+
\varepsilon\log\bigg(1+\frac{1}{4\varepsilon^2}\bigg).
$$
Note that $d_\varepsilon \to 0$ as $\varepsilon\to 0$. Set
$$
M=\frac{\frac{\pi}{2}+\eta+1}{2\varepsilon}.
$$
Since  $\varepsilon\log K<\eta<\frac{\pi}{2}+\eta+1$,
$$
M>\frac{\log K}{2}.
$$
Define the following expression on the entire disk by
$$
O(z)=
\begin{cases}
\mathcal E(U;F)-Mx(z),&\text{on } \{z\in\mathbb{D}:x(z)>0\},\\
\mathcal E(V;-F)+Mx,&\text{on } \{z\in\mathbb{D}:x(z)<0\},\\
-\frac{\pi}{2}|y(z)|,&\text{on } \{z\in\mathbb{D}:x(z)=0\}.
\end{cases}
$$
The last line includes the value zero at the zeros of $F$. Lemma \ref{entropy-lemma} shows that $O$ is continuous. On each open set $\{z\in\mathbb{D}:x(z)>0\}$ or $\{z\in\mathbb{D}:x(z)<0\}$, $O$ has local harmonic upper supports.

Furthermore, on $\{z\in\mathbb{D}:x(z)>0\}$, the lower bound of \eqref{pointwise estimates} and $h\le\log U$ give
\begin{equation}\label{lower O_1}
O\ge x\log U-Mx-\frac{\pi}{2}|y|
\ge xh-M|x|-\frac{\pi}{2}|y|.
\end{equation}
On $\{z\in\mathbb{D}:x(z)<0\}$, the real part of $-F$ is $-x>0$. Since $\log V\ge-h$, 
\begin{equation}\label{lower O_2}
O\ge(-x)\log V+Mx-\frac{\pi}{2}|y|
\ge xh-M|x|-\frac{\pi}{2}|y|.
\end{equation}
The same lower bound is immediate on $\{z\in\mathbb{D}:x(z)=0\}$.

Now, consider the open cone
$$
\mathcal C=\{z\in\D:|x(z)|<2\varepsilon|y(z)|\}.
$$
Here $y$ cannot vanish, so $\tau=\sgn(y)$ is locally constant by continuity. The functions
$$
Z_+=\bigg(\frac{1}{2}-i\tau\varepsilon\bigg)F
\ \ \text{ and }\ \ 
Z_-=\bigg(-\frac{1}{2}-i\tau\varepsilon\bigg)F
$$
are therefore holomorphic on each component of the cone. Write $Z_\pm=p_\pm+iq_\pm$. Their coordinates are
$$
\begin{aligned}
p_+&=\frac{x}{2}+\varepsilon|y|,
&
p_-&=-\frac{x}{2}+\varepsilon|y|,\\
q_+&=\frac{y}{2}-\tau\varepsilon x,
&
q_-&=-\frac{y}{2}-\tau\varepsilon x.
\end{aligned}
$$
The cone condition gives $p_+,p_->0$ and $p_+-p_-=x$. Moreover, since $0<\varepsilon<\frac{1}{4}$,
$$
|\varepsilon x|<2\varepsilon^2|y|<\frac{|y|}{2}.
$$
Consequently $q_+$ has the sign of $y$, $q_-$ has the opposite sign, and
$$
|q_+|+|q_-|=|y|.
$$

Define
\begin{displaymath}
Q=\mathcal E(U;Z_+)+\mathcal E(V;Z_-)-\eta|y|
\quad
\text{on }\mathcal C.
\end{displaymath}
Since $y$ has constant sign locally, $|y|$ is harmonic there. This combined with Lemma \ref{entropy-lemma}, there is a touching harmonic upper support for $Q$. Inequality \eqref{pointwise estimates} gives,
$$
Q
\ge
p_+\log U+p_-\log V-(\frac{\pi}{2}+\eta)|y|.
$$
Recall that $p_+-p_-=x$, so
$$
p_+\log U+p_-\log V-xh
=
p_+(\log U-h)+p_-(\log V+h)  \ge0.
$$
thus, using the hypothesis that $-\log V\le h\le\log U$, it follows that
\begin{equation}\label{lower Q}
Q\ge xh-(\frac{\pi}{2}+\eta)|y|.
\end{equation}
Now for a point where $x(z)=0$ and $t=|y(z)|>0$, one has
$$
Z_+=\varepsilon t+i\tau\frac{t}{2},
\ \ \ \
Z_-=\varepsilon t-i\tau\frac{t}{2},
\ \ \text{and }\ \
O=-\frac{\pi}{2}t.
$$
Therefore, formula \eqref{epsilon} gives (using $\frac{\pi}{2}-\arctan\frac{1}{2\varepsilon}
=
\arctan(2\varepsilon)$ and that $UV\le K$)
$$
\begin{aligned}
Q-O
&=
t\varepsilon\log(UV)
+
t\varepsilon\log\bigg(1+\frac{1}{4\varepsilon^2}\bigg)
+
t\bigg(\frac{\pi}{2}-\arctan\frac{1}{2\varepsilon}-\eta\bigg)\\
&=
t(\varepsilon\log(UV)+d_\varepsilon-\eta)
<0.
\end{aligned}
$$
Hence $Q<O$ on $\{z\in\mathbb{D}:x(z)=0, |y(z)|>0\}$.

At an edge of the cone, $|x(x)|=2\varepsilon|y(z)|>0$, and one of $p_+,p_-$ equals zero while the other equals $|x|$. Put $R_x=U$ for $x>0$ and $R_x=V$ for $x<0$. Inequality \eqref{pointwise estimates}, including their boundary values, give
$$
Q\ge|x|\log R_x-(\frac{\pi}{2}+\eta)|y|
\ \ \text{ and } \ \
O\le|x|\log R_x-M|x|.
$$
Thus
$$
Q-O
\ge
(2\varepsilon M-\frac{\pi}{2}-\eta)|y|
=
|y|
>
0.
$$
That is, $Q>O$ on the cone edges.

Define 
$$
B=
\begin{cases}
\min\{O,Q\},&z\in\mathcal C,\\
O,&z\notin\mathcal C.
\end{cases}
$$
A zero of $F$ lies outside the open cone, and the definition of $O$ gives $B=0$ there.

Near a nonzero point of $F$ with $x(z)=0$, continuity and $Q<O$ imply $B=Q$. Near a nonzero cone-edge point, continuity and $O<Q$ imply $B=O$, including on the portion of its neighborhood inside the cone. At every other nonzero point, $B$ is locally either $O$ (if $z\notin\mathcal{C}$) or the minimum of $O$ and $Q$ (if $z\in\mathcal{C}$). This proves continuity away from zeros of $F$.

These local descriptions also supply harmonic upper supports away from zeros. For a minimum, choose a branch attaining the minimum at the point and take a touching upper support for that branch. The support dominates the minimum throughout a neighborhood and touches it at the chosen point.

The lower estimates \eqref{lower O_1}, \eqref{lower O_2}, \eqref{lower Q} for $O$ and $Q$ imply
$$
B\ge hx-M|x|-(\frac{\pi}{2}+\eta)|y|,
$$
while \eqref{pointwise estimates} gives,
\begin{equation}\label{ineq for B}
B\le O\le |x|(|\log U|+|\log V|).
\end{equation}
Near a zero $z_0$ of $F$, the functions $h,\log U,\log V$ are bounded. As $z$ tends to $z_0$, squeeze theorem implies that $\lim\limits_{z\to z_0}B(z)=0=B(z_0)$. Hence $B$ is continuous at every zero of $F$.

To show that $B$ is superharmonic on the set of zeros of $F$, define
$$
c=\frac{\log U(z_0)-\log V(z_0)}{2}.
$$
Because $M>\frac{\log K}{2}$ and $U(z_0)V(z_0)\le K$,
\begin{displaymath}
\log U(z_0)-M<c<M-\log V(z_0).
\end{displaymath}
By continuity, the same strict inequalities hold in a neighborhood of $z_0$. If for some $z$, $x(z)>0$ in that neighborhood, the definition of $O$ and \eqref{ineq for B} give
$$
B\le O\le x(\log U-M)\le cx, 
$$
while if $x(z)<0$,
$$
B\le O\le x(M-\log V)\le cx.
$$
Finally, on $\{z\in\mathbb{D}:x(z)=0\}$, the inequality is $B\le O=-\frac{\pi}{2}|y|\le0=cx$. Thus the harmonic function $c\Real F$ dominates $B$ near $z_0$ and equals $B$ at $z_0$.

Therefore, we have proved continuity and constructed a touching harmonic upper support at every point. Lemma \ref{upper-supports} gives superharmonicity on all of $\D$. 
\end{proof}

Let's now build the last tool before we prove Theorem \ref{main}.

\begin{proposition}
\label{graph-estimate}
Let $w\in \mathcal{A}_2(\T)$. For every $\eta>0$ there is a finite $M$, depending only on $[w]_{\mathcal{A}_2}$ and $\eta$, such that
\begin{equation}\label{basic ineq}
\bigg|\int_\T f\log w\,dm\bigg|
\le
M\|f\|_1+
\bigg(\frac{\pi}{2}+\eta\bigg)\|Hf\|_1
\end{equation}
for every real mean-zero trigonometric polynomial $f$.
\end{proposition}

\begin{proof}
The elementary inequality $|\log w|\le w+w^{-1}$ gives $\log\ w\in L^1$. Define
$$
h=P(\log w),
\quad
U=Pw,
\ \ \ \text{ and } \
V=P(w^{-1}).
$$
The Poisson extensions are harmonic, and $U,V$ are strictly positive. Jensen's inequality for the probability measure $P_z\,dm$ gives
$$
h=P(\log w)\le\log(Pw)=\log U.
$$
Jensen's inequality also gives $-h\le\log V$, as $\log(w^{-1})=-\log\ w$. Thus, $-\log V\le h\le\log U$. By Lemma \ref{poisson-A2},
$$
1\le UV\le K=\frac{512}{3}[w]_{\mathcal{A}_2}^2.
$$

For a real mean-zero trigonometric polynomial $f$, note that its Fourier coefficients satisfy $\widehat{f}(0)=0$, $\widehat{f}(-n)=\overline{\widehat{f}(n)}$ and 
$
\widehat{(f+i Hf)}(n)=2\widehat{f}(n)
$
for $n>0$ and zero otherwise. 

Now set
\begin{displaymath}
F(z)=2\sum_{n=1}^N\widehat f(n)z^n.
\end{displaymath}
The sum is finite and $F(0)=0$. Also, $F$ has the boundary values $f(e^{it})+iHf(e^{it})$. Since both $f$ and $Hf$ are real valued, $\textrm{Re} F(e^{it})=f(e^{it})$ and $\textrm{Im} F(e^{it})=Hf(e^{it})$. Thus, if $F(z)=x(z)+iy(z)$, we have
\begin{displaymath}
x(z)=Pf(z)
\text{ and }
y(z)=P(Hf)(z).
\end{displaymath}

Let $B$ be the superharmonic function from Lemma \ref{gluing}. Since $F(0)=0$, we have $B(0)=0$. Inequality \eqref{lower bnd} and lemma \ref{upper-supports} give, for $0<r<1$,
\begin{equation}\label{negative int}
\int_\T
\left(h(r\zeta)x(r\zeta)-M|x(r\zeta)|-\left(\frac{\pi}{2}+\eta\right)|y(r\zeta)|\right)
\,dm(\zeta)
\le
\int_\T B(r\zeta)\,dm(\zeta)
\le  B(0)=0.
\end{equation}

Write $h_r(\zeta)=h(r\zeta)=(P_r*\log\ w)(\zeta)$ and similarly for $x_r=P_r*f$ and $y_r=P_r*Hf$. As $r$ increases to one, the finite Fourier sums give uniform convergence: $x_r\longrightarrow f$ and $y_r\longrightarrow Hf$. Also, $h_r\to \log\ w$ in $L^1$. Notice that
$$
\bigg|\int_\T h_rx_r\,dm-\int_\T f\log w\,dm\bigg|
\le
\|h_r-\log w\|_1\|x_r\|_\infty
+
\|\log w\|_1\|x_r-f\|_\infty
\longrightarrow 0.
$$
Here $\|x_r\|_\infty\le \|f\|_\infty $. Therefore, $\displaystyle \int_\T h_rx_r\,dm\to \int_\T f\log w\,dm$ as $r\to 1$, while from \eqref{negative int},
\begin{equation}\label{int hr}
   \int_\T h_rx_r\,dm \le M\int_{\mathbb{T}}|x_r|dm+\left(\frac{\pi}{2}+\eta\right)\int_{\mathbb{T}}|y_r|dm
\end{equation}
Since $x_r\to f$ uniformly and $||x_r|-|f||\le |x_r-f|$, we get $\displaystyle \int_{\mathbb{T}}|x_r|dm\to \Vert f\Vert_1$. Similarly, $\displaystyle\int_{\mathbb{T}}|y_r|dm\to\Vert Hf\Vert_1$. Hence, for $r\to 1$, \eqref{int hr} implies
$$
\int_\T f\log w\,dm
\le
M\|f\|_1+\left(\frac{\pi}{2}+\eta\right)\|Hf\|_1.
$$
Repeating the same argument with $-f$ yields $\displaystyle 
\int_\T (-f)\log w\,dm
\le
M\|f\|_1+\left(\frac{\pi}{2}+\eta\right)\|Hf\|_1.
$
The two last inequalities combined imply \eqref{basic ineq}.  
\end{proof}

\section{Proof of Theorem \ref{main}}
\label{sec:forward}

\begin{proposition}\label{thm:HS}
Let $w$ be an $\mathcal{A}_2$ weight on \(\mathbb{T}\). Then there exist real-valued functions $u,v\in L^\infty(\mathbb{T})$ such that
\[
\log w=u+Hv
\text{ with }
\Vert v \Vert_\infty<\frac{\pi}{2}.
\]
\end{proposition}

\begin{proof}
Let $w\in \mathcal{A}_2$. Lemma \ref{self-improvement} gives $\delta>0$ such that $w^{1+\delta}\in \mathcal{A}_2$. Because $\log(w^{1+\delta})=(1+\delta)\log w$, Proposition \ref{graph-estimate} to $w^{1+\delta}$ with
\begin{displaymath}
\eta=\frac{\pi\delta}{4}
\end{displaymath}
gives that there is a finite $M_\delta$ such that
\begin{displaymath}
(1+\delta)\bigg|\int_\T f\log w\,dm\bigg|
\le
M_{\delta}\|f\|_1+
\bigg(\frac{\pi}{2}+\frac{\pi\delta}{4}\bigg)\|Hf\|_1.
\end{displaymath}
Enlarging the first coefficient if necessary, choose $C>0$ such that
\begin{displaymath}
\bigg|\int_\T f\log w\,dm\bigg|
\le
C\|f\|_1+\alpha\|Hf\|_1,
\qquad
\alpha=\frac{\pi(\delta+2)}{4(\delta+1)}.
\end{displaymath}
Note that
\begin{displaymath}
\frac{\pi}{2}-\alpha
=
\frac{\pi\delta}{4(\delta+1)}
>
0.
\end{displaymath}

Equip now the real Banach space
\begin{displaymath}
X=L^1(\T;\mathbb R)\oplus L^1(\T;\mathbb R)
\end{displaymath}
with the norm $\|(g,k)\|_X=C\|g\|_1+\alpha\|k\|_1$. Consider its linear subspace
\begin{displaymath}
Y=\{(f,Hf):f\text{ is a real mean-zero trigonometric polynomial}\}.
\end{displaymath}
Define $\Lambda$ on $Y$ by
\begin{displaymath}
\Lambda(f,Hf)=\int_\T f\log w\,dm.
\end{displaymath}
The first coordinate uniquely determines every point of $Y$, so $\Lambda$ is well-defined. Inequality \eqref{basic ineq} gives precisely that $|\Lambda(z)|\le\|z\|_X$ for $z\in Y$.

The real Hahn--Banach theorem gives an extension $\widetilde\Lambda$ to $X$ with the same bound. The subspace $Y$ need not be closed. Restricting the extension to each coordinate gives
\begin{displaymath}
|\widetilde\Lambda(g,0)|\le C\|g\|_1
\ \ \text{ and }\ \
|\widetilde\Lambda(0,k)|\le\alpha\|k\|_1.
\end{displaymath}
Real $L^1$ duality therefore provides real $u_0,a_0\in L^\infty$ satisfying
\begin{displaymath}
\widetilde\Lambda(g,0)=\int_\T u_0g\,dm
\ \ \text{ and }\ \
\widetilde\Lambda(0,k)=\int_\T a_0k\,dm,
\end{displaymath}
with $\|u_0\|_\infty\le C$ and $\|a_0\|_\infty\le\alpha$. By linearity of $\widetilde{\Lambda}$ and agreement with $\Lambda$ on $Y$,
\begin{displaymath}
\int_\T f\log w\,dm
=
\int_\T u_0f\,dm+
\int_\T a_0Hf\,dm
\end{displaymath}
for every real mean-zero trigonometric polynomial.

Since $a_0\in L^\infty(\mathbb{T})$, we have $a_0\in L^2(\mathbb{T})$, and $Ha_0\in L^2(\mathbb{T})\subset L^1(\mathbb{T})$. Skew-adjointness gives
\begin{displaymath}
\int_\T a_0Hf\,dm=-\int_\T(Ha_0)f\,dm.
\end{displaymath}
Consequently the real $L^1$ function
\begin{displaymath}
q=\log w-u_0+Ha_0
\end{displaymath}
annihilates every real mean-zero trigonometric polynomial. For each $n\ge1$, the functions $\cos(nt)$ and $\sin(nt)$ are admissible test functions, so
$\int_\T q\cos(nt)\,dm=\int_\T q\sin(nt)\,dm=0$.
Since $e^{\pm int}=\cos(nt)\pm i\sin(nt)$, all nonconstant Fourier coefficients of $q$ vanish. Set $c=\int_\T q\,dm$. Then $q-c$ also has zero constant coefficient, so all its Fourier coefficients are zero. Every Fej\'er mean of $q-c$ is a finite sum of these coefficients and is therefore identically zero. Since the Fej\'er means converge to $q-c$ in $L^1$, we obtain $\|q-c\|_1=0$, and hence $q=c$ almost everywhere.

Thus
\begin{displaymath}
\log w=u_0-Ha_0+c.
\end{displaymath}
Set $u=u_0+c$ and $v=-a_0$. We obtain that $u,v\in L^\infty(\T)$ and
\begin{displaymath}
w=e^{u+Hv},
\text{ with }
\|v\|_\infty\le\alpha<\frac{\pi}{2}.
\end{displaymath}
This concludes the proof of the Proposition \ref{thm:HS}.
\end{proof}


\begin{proposition}\label{converse}
The Helson-Szeg\H{o} condition \eqref{helsz} implies the $\mathcal{A}_2$ condition \eqref{A2 weights}.    
\end{proposition}
\begin{proof}
 Suppose that $\log w =u +Hv$, where $u,v \in L^{\infty}(\mathbb{T})$ with $\beta=\Vert v \Vert_\infty<\frac{\pi}{2}$.

Define
\begin{displaymath}
L(z)=P(Hv)(z)-iPv(z),
\end{displaymath}

Using that $\displaystyle Pv(re^{it})=\sum_{n\in\mathbb{Z}}r^{|n|}\widehat{v}(n)e^{int}$ and that $\widehat{Hv}(n)=-i\text{sgn}(n)\widehat{v}(n)$, we get that 
\begin{displaymath}
L(z)
=
-i\widehat v(0)
-
2i\sum_{n\ge1}\widehat v(n)z^n.
\end{displaymath}
Note that, since $v\in L^2(\mathbb{T})$, this power series converges locally uniformly in $\D$. Also note that $L$ is holomorphic on $\mathbb{D}$. Define now the holomorphic function $G(z)=e^{L(z)}$.

Since the Poisson kernel is positive and its integral equals 1, we obtain $|Pv|\le\Vert v\Vert_\infty=\beta$. Thus, $|\Imag L|\le\beta$ and this implies that 
both $G$ and $G^{-1}$ lie in the sector of half-angle $\beta$ about the positive real axis. In particular, they have positive real part and satisfy, for every $z\in\mathbb{D}$, 
\begin{equation}\label{sector ineq}
|G(z)|\le\sec\beta\,\Real G(z)
\ \ \text{ and }\ \
|G(z)|^{-1}\le\sec\beta\,\Real(G(z)^{-1}).
\end{equation}
For example, the first inequality follows by noticing that $\Real G=e^{\Real L}\cos(\Imag L)$ and since cosine decreases on $[0,\frac{\pi}{2}]$ we have $\Real G \ge |G|\cos\beta$. Similarly, the second one follows.

For each fixed $r<1$, we have
\begin{displaymath}
\int_\T |G(r\zeta)|\,dm(\zeta)
\le
\sec\beta\int_\T \Real G(r\zeta)\,dm(\zeta)
=
\sec\beta\,\Real G(0),
\end{displaymath}
by the mean-value property of harmonic functions. The identical argument applied to $G^{-1}$ gives
\begin{displaymath}
\int_\T |G(r\zeta)|^{-1}\,dm(\zeta)
\le
\sec\beta\,\Real(G(0)^{-1}).
\end{displaymath}
Note that the right-hand sides are finite and independent of $r$. Also, $G(0)=e^{-i\int_\T v\,dm}<\infty$.

The Poisson extensions of $Hv\in L^2(\mathbb{T})$ and $v\in L^\infty(\mathbb{T})$ converge almost everywhere to their boundary functions. Putting $g(\zeta)=e^{Hv(\zeta)}$, we have
\begin{displaymath}
|G(r\zeta)|\longrightarrow g(\zeta)
\ \ \text{ and }\ \
|G(r\zeta)|^{-1}\longrightarrow g(\zeta)^{-1}
\end{displaymath}
almost everywhere as $r\to 1^{-}$. Applying Fatou's lemma to these nonnegative functions proves
$$
\int_\T g\,dm\le\sec\beta\,\Real G(0)<\infty
\ \ \text{ and }\ \
\int_\T g^{-1}\,dm\le\sec\beta\,\Real(G(0)^{-1})<\infty.
$$
That is, $e^{Hv}, e^{-Hv}\in L^1(\mathbb{T})$. This implies that $w, w^{-1}\in L^1(\mathbb{T})$. In fact, since $u$ is bounded,
$$
\int_T w\, dm =\int_T e^ug\, dm \le e^{\Vert u\Vert_\infty} \int_T g\, dm \le e^{\Vert u\Vert_\infty}\sec\beta\,\Real G(0)
$$
and similarly we obtain that $w^{-1}\in L^1(\mathbb{T})$.

Fix $z\in\D$ and choose $|z|<r<1$. The Poisson formula for harmonic functions on the disk of radius $r$ gives
$$
\Real G(z)
=
\int_\T P_{z/r}(\zeta)\Real G(r\zeta)\,dm(\zeta).
$$
Combining this with inequality \eqref{sector ineq}, we obtain
$$
\int_\T P_{z/r}(\zeta)|G(r\zeta)|\,dm(\zeta)
\le
\sec\beta\,\Real G(z).
$$
As $r\to 1$, $P_{z/r}\to P_z$ and $|G(r\zeta)|\to g(\zeta)$ almost everywhere.  Fatou's lemma gives that
$$
\int_{\mathbb{T}}P_z(\zeta)g(\zeta)dm(\zeta)\le \sec\beta\,\Real G(z)
$$
Therefore,
$$
Pg(z)\le\sec\beta\,\Real G(z).
$$
The same proof for $G^{-1}$ gives
$$
P(g^{-1})(z)\le\sec\beta\,\Real(G(z)^{-1}).
$$
Thus,
$$
Pg(z)P(g^{-1})(z)\le\sec^2\beta.
$$
since $\Real(G(z))\Real(G(z)^{-1})=\cos^2(\Imag L(z))\le 1$.

Let $\varphi=\operatorname{ess\,sup}u-\operatorname{ess\,inf}u$. Positivity of $g$ and the Poisson kernel and boundedness of $u$ give
$$
Pw(z)\le e^{\operatorname{ess\,sup}u}Pg(z)
\ \ \text{ and }\ \
P(w^{-1})(z)\le e^{-\operatorname{ess\,inf}u}P(g^{-1})(z).
$$
Therefore,
$$
Pw(z)P(w^{-1})(z)\le e^\varphi\sec^2\beta.
$$
Lemma \ref{poisson-A2} proves $w\in \mathcal{A}_2$, completing Proposition \ref{converse}. 
\end{proof}

To conclude, Propositions \ref{thm:HS} and \ref{converse} combined prove Theorem \ref{main}.

\hspace{1cm}

\end{document}